\documentclass[12pt, reqno]{amsart}
\usepackage{varioref}
\usepackage{hyperref}
\usepackage{thmtools}
\usepackage{comment}
\usepackage[capitalize,nameinlink,noabbrev]{cleveref}
\usepackage{longtable}
\usepackage{caption}
\usepackage{booktabs}
\usepackage{makecell}
\usepackage{stmaryrd}
\expandafter\def\csname opt@stmaryrd.sty\endcsname
{only,shortleftarrow,shortrightarrow}
\usepackage{extpfeil}
\usepackage{amsmath,amssymb,amsthm}
\usepackage{hyperref}
\hypersetup{colorlinks=true,urlcolor=blue,citecolor=blue,linkcolor=blue}
\usepackage{courier}
\usepackage{tikz}
\usepackage{tikz-cd}
\usetikzlibrary{calc,matrix,arrows,decorations.markings}
\usepackage{array}
\usepackage{color}
\usepackage{enumerate}
 \usepackage{cancel}
\usepackage{nicefrac}
\usepackage{listings}
\usepackage{seqsplit}
\usepackage{colonequals}
\newcommand{\F}{\mathbb{F}}
\newcommand{\Q}{\mathbb{Q}}

\newcommand{\C}{\mathbb{C}}

\newcommand{\Z}{\mathbb{Z}}

\newcommand{\Aut}{{\rm Aut}}

\newcommand{\Sym}{\operatorname{Sym}}

\newcommand{\id}{\operatorname{id}}

\renewcommand{\P}{\mathbb{P}}

\newcommand{\gon}{\operatorname{gon}}
\newcommand{\Fix}{\operatorname{Fix}}

\DeclareFontFamily{U}{wncy}{}
    \DeclareFontShape{U}{wncy}{m}{n}{<->wncyr10}{}
    \DeclareSymbolFont{mcy}{U}{wncy}{m}{n}
    \DeclareMathSymbol{\Sh}{\mathord}{mcy}{"58}

\newtheorem{theorem}{Theorem}
\theoremstyle{definition}

\newtheorem{remark}[theorem]{Remark}
\newtheorem{proposition}[theorem]{Proposition}

\newtheorem{lemma}[theorem]{Lemma}
\newtheorem{corollary}[theorem]{Corollary}

\usepackage{xcolor}

\newcolumntype{L}[1]{>{\raggedright\arraybackslash}p{#1}}
\renewcommand{\arraystretch}{1.3}

\begin{document}

\author{Oana Padurariu}
\address{Oana Padurariu \\
Max-Planck-Institut für Mathematik Bonn\\
Germany}
\urladdr{https://sites.google.com/view/oanapadurariu/home}
\email{opadurariu@mpim-bonn.mpg.de}

\title[Modular curves $X_0(N)$ and Shimura curves $X_0^D(N)$ that embed in $\P^1 \times \P^1$]{Modular and Shimura curves $X_0^D(N)$ embedded in $\P^1 \times \P^1$}

\begin{abstract}
We classify the modular curves $X_0(N)$ that admit a smooth embedding in $\P^1 \times \P^1$ and obtain a partial classification for the Shimura curves $X_0^D(N)$. We also prove that, for $D>1$, a Shimura curve $X_0^D(N)$ admits a smooth plane model if and only if its genus is at most one.
\end{abstract}

\maketitle

\section{Introduction}

Finding suitable models for curves has been a rich area of research and is crucial when trying to study their properties, including their set of points defined over a number field. The modular curves of Shimura type that admit a smooth plane model were determined by Assaf--Anni-Lorenzo-García \cite{AALG23}. A natural question is which modular and Shimura curves embed  \footnote{All curves and surfaces are considered over $\C$ unless otherwise specified.
In particular, by $\P^1,\P^1 \times \P^1$ and $\P^2$ we mean $\P^1_{\C}$,$\P^1_{\C} \times \P^1_{\C}$ and $\P^2_{\C}$, respectively.} into other smooth surfaces.
In this paper we consider smooth quadrics, which over $\C$ are all isomorphic to $\P^1 \times \P^1$. Since a point on $X_0(N)$  parametrizes an elliptic curve $E$ together with a cyclic isogeny $\iota : E \to E'$ with kernel cyclic of order $N$, the assignment
\[
(E,\iota) \mapsto (j(E),j(E'))
\]
gives a map $X_0(N) \to \P^1 \times \P^1$, but this is not necessarily a closed immersion. In this paper we prove that there are finitely many pairs $(D,N)$ with discriminant $D$ and level $N$, where $D \ge 1$ is a product of an even number of distinct primes and $N$ is coprime to $D$ such that $X_0^D(N)$ admits a smooth model in $\P^1 \times \P^1$. 

The Magma code accompanying the paper can be found in \cite{Repo}. The main result for modular curves is the following:

\begin{theorem}
The modular curve $X_0(N)$ geometrically embeds in  $\P^1 \times \P^1$ if and only if $\gon_{\C}(X_0(N)) \le 2$, or $X_0(N)$ is a genus $4$ curve with $N \in \{38,44,53,54,61 \}$ with models over $\P^1 \times \P^1$ provided in \cref{tab: P1twice_models}.
\end{theorem}

Shimura curves of genus at most $2$ were enumerated by Voight \cite{Voight09}, Gonzalez--Rotger \cite{GR06} provided defining equations for $X_0^D(N)$ of genus $1$, while Guo--Yang \cite{GY17} computed equations for the geometrically hyperelliptic ones.
So our analysis starts with $\gon_{\C}(X_0^D(N)) \ge 3$, leading to the main result for Shimura curves:

\begin{theorem}\label{thm:shimura-p1xp1}
Let $D>1$ be the discriminant of an indefinite quaternion algebra
over $\mathbb{Q}$, let $N\geq 1$ be coprime to $D$, and put
$X \coloneq X_0^D(N)$.
\begin{enumerate}
    \item If $\operatorname{gon}_{\mathbb{C}}(X)\leq 2$, or
    \[
        (D,N)\in\{(6,35),(10,21),(14,17)\},
    \]
    then $X$ admits a geometric embedding in
    $\mathbb{P}^1\times\mathbb{P}^1$.
    In the three listed cases, $X$ has genus $9$ and admits
    an embedding of bidegree $(4,4)$.

    \item Conversely, if $X$ admits a geometric embedding in
    $\mathbb{P}^1\times\mathbb{P}^1$ and
    $\operatorname{gon}_{\mathbb{C}}(X)\geq 3$, then $(D,N)$
    is one of these three pairs or one of the eleven pairs
    in \cref{table: remaining_candidates_Shimura}.
\end{enumerate}
\end{theorem}

For the eleven pairs in \cref{table: remaining_candidates_Shimura},
the existence of such an embedding remains unresolved;
the table records the possible bidegrees.

\renewcommand{\arraystretch}{1.15}
{
\begin{longtable}{|c|c|}
\caption{All $11$ pairs $(D,N)$ for which we remain unsure whether $X_0^D(N)$ embeds in $\P^1 \times \P^1$.}\label{table: remaining_candidates_Shimura} \\  \hline 
$(a,b)$ & $(D,N)$  \\ \hline \hline
$(3,3)$ & $(106,1)$\footnote{See \cref{remark: conjecture106}},$(118,1)$ \footnote{See \cref{remark: conjecture118}} \\ \hline
$(4,4)$ & $(38,5),(133,1),(145,1),(177,1),(226,1)$ \\ \hline
$(4,6)$ & $(217,1),(267,1),(382,1)$ \\ \hline
$(5,5)$ & $(394,1)$ \\ \hline 
\end{longtable}
}

The proofs consist of two main steps. We first reduce the search to a finite set using the relation between genus and gonality of $X_0(N)$ and $X_0^D(N)$. For the remaining modular candidates we either find a model in $\P^1 \times \P^1$ or prove that no such model exists. For the remainder of the Shimura curves, we either prove that a model in $\P^1 \times \P^1$ exists (without providing defining equations), or we demonstrate its non-existence. Unlike modular curves, Shimura curves have no cusps and no real points, and defining equations are correspondingly scarcer in the literature.
We also discuss a companion result about embeddings in another toric surface, namely $\P^2$. In the spirit of \cite{AALG23}, we also prove that

\begin{theorem}\label{theorem:main_theorem}
For $D> 1$ the Shimura curve $X_0^D(N)$ admits a smooth plane model in $\P^2$ if and only if the genus of $X_0^D(N)$ is at most one.
\end{theorem}

\subsection*{Acknowledgements}
The author would like to thank John Voight for suggesting the project, and is grateful to Gebhard Martin, Freddy Saia, and John Voight for helpful conversations. 

\subsection*{AI Statement} Some of the propostions and lemma in this paper were produced in conversation with ChatGPT 5.6 (\cref{lem:fixed-point-obstruction-P1xP1}, \cref{lemma: 2a_obstruction}, \cref{lem:low-gonality-embedding}, \cref{prop: 14-17)}), as well as \cref{remark: conjecture118}. ChatGPT 5.6 was also used to produce the defining equations in \cref{tab: P1twice_models}, as well as to proofread the paper. The author takes full responsibility for the paper, the code, and their mathematical correctness.

\section{Gonality bounds}

The following bound was first given by Abramovich and later improved by Kim--Sarnak:

\begin{theorem}{\cite[Theorem 1.1]{Abramovich96}, \cite[Appendix 2]{KS03}}
\label{theorem: index_gonality_inequality}
Let $\Gamma \subset \text{PSL}_2(\Z)$ be a congruence subgroup, and $X_{\Gamma}$ the corresponding modular curve. Then
\[
\gon_{\C}(X_{\Gamma}) \ge \frac{1}{24}\cdot\frac{975}{4096}[\text{PSL}_2(\Z) : \Gamma].
\]
\end{theorem}

To obtain an upper bound on the genus of $X_0^D(N)$ we use the following weaker result which works for both modular and Shimura curves:
\begin{theorem}{\cite[Theorem 1.1]{Abramovich96}, \cite[Appendix 2]{KS03}} \label{theorem: genus_gonality_inequality}
There is an inequality:
\[
\frac{1}{2} \cdot \frac{975}{4096} \left( g(X_0^D(N))-1\right) \le \gon_{\C} (X_0^D(N)).
\]   
\end{theorem}

By $W_0(D;N)$ we denote the full group of Atkin--Lehner involutions. We will use the following notation throughout the paper:

\[
X_0^D(N)^+ \coloneq X_0^D(N)/\langle w_{D\cdot N} \rangle, \quad X_0^D(N)^* \coloneq X_0^D(N)/W_0(D;N).
\]

To determine the properties of $X_0^D(N)$, one may appeal to its quotients by involutions, for which knowing the number of fixed points is necessary and sufficient to compute the genus of the quotient:

\begin{proposition}[Riemann--Hurwitz]
Let $\sigma$ be any involution on a smooth projective curve $C$ over an algebraically closed field $k$ of characteristic $0$, and let $\#\Fix_C(\sigma)$ denote the number of fixed points of $\sigma$. Then, we have the following formula for the genus of the quotient:
$$g(C/\langle \sigma \rangle) = \frac{1}{4}(2g(C)+2- \#\Fix_C(\sigma)).$$
\end{proposition}

Shimura curves $X_0^D(N)$ have a rich automorphism group due to their Atkin--Lehner involutions. Ogg proved that only certain CM points can be fixed by these involutions, while also giving a formula to compute their number. Let $\mathcal{O}_N$ be an Eichler order of level $N$ in the indefinite quaternion algebra over $\Q$ of discriminant $D$. For each prime $p$, embeddings into $(\mathcal{O}_N)_p$ are counted up to conjugation by $(\mathcal{O}_N)_p^{\times}$.

\begin{theorem}\cite[p. 283, Theorem 2]{Ogg83}\label{thm: Ogg_fixed_pts}
Let $D>1$ and $m>1$ with $m \parallel DN$. The fixed points of the Atkin--Lehner involution $w_m$ acting on $X_0^D(N)$ are points with \textnormal{CM} by the following imaginary quadratic orders:
\[   
R = 
     \begin{cases}
       \Z[i] \text{ and } \Z[\sqrt{-2}] &\quad\textnormal{if } m = 2, \\
       \Z\left[\frac{1+\sqrt{-m}}{2}\right]\text{and } \Z[\sqrt{-m}] &\quad\textnormal{if } m \equiv 3 \;(\bmod \; 4), \\
       \Z[\sqrt{-m}] &\quad\textnormal{ otherwise}.\\
     \end{cases}
\]

For each of the orders $R$ listed for $m$ in this theorem, the count of $R$-CM points fixed by $w_m$ is given explicitly by Ogg as a product
\[
h(R) \prod_{\substack{p \mid \frac{DN}{m} \\ \text{prime}}} \nu_p(R,\mathcal{O}_N),
\]
where $h(R)$ is the class number of the order $R$ and $\nu_p(R,\mathcal{O}_N)$ is the number of inequivalent optimal embeddings of the localization $R_p$ of $R$ at $p$ into the localization $(\mathcal{O}_N)_p$  of $\mathcal{O}_N$ at $p$.
\end{theorem}

The following result is fundamental in determining the relation between the gonality of a curve and those of its quotients:

\begin{theorem}[Castelnuovo--Severi]\label{theorem: CS}
Let $F$ be a perfect field and let $C$, $C_1$, and $C_2$ be curves over $F$. Let
\[
\pi_1\colon C \to C_1\qquad\text{and} \qquad \pi_2\colon C \to C_2
\]
be non-constant morphisms defined over $F$. Then either
\[
g(C) \leq \textnormal{deg}(\pi_1)\cdot g(C_1) + \textnormal{deg}(\pi_2)\cdot g(C_2) + (\textnormal{deg}(\pi_1)-1)(\textnormal{deg}(\pi_2)-1),
\]
or there exists a curve $C'$ over $F$ and a morphism $C \to C'$ over $F$ of degree greater than $1$ through which both $\pi_1$ and $\pi_2$ factor. 
\end{theorem}

To rule out that a curve embeds in $\P^1 \times \P^1$ we can use the following consequences of the Castelnuovo--Severi inequality:

\begin{proposition}
Let $C/\C$ be a curve of genus $(a-1)(b-1)$ with $a \le b$ that admits an involution $\sigma$. Assume 
\[
    g(C) > 2 g(C/\langle \sigma \rangle) + (b-1).
\]
    
Then $C$ cannot embed in $\P^1 \times \P^1$ with bidegree $(a,b)$.    
\end{proposition}

\begin{proof}
Assume for a contradiction that $C$ embeds in $\P^1 \times \P^1$ with bidegree $(a,b)$. We know that 

\[
    g(C) > b\cdot 0 + 2 g(C/\langle \sigma \rangle) + (b-1)(2-1) \ge a \cdot 0 + 2 g(C/\langle \sigma \rangle) + (a-1)(2-1),
\]
hence by the Castelnuovo--Severi inequality we know that the $a$-gonal map and $\sigma$ factor through a map. Moreover, the $b$-gonal map and $\sigma$ factor through a map. But since $C \mapsto C/\langle \sigma \rangle$ is a Galois cover, it follows that both the $a$-gonal and the $b$-gonal map factor through the map $C \mapsto C/\langle \sigma \rangle$, contradicting the fact that these two maps are independent.

\end{proof}

\begin{proposition}
Let $C/\C$ be a curve of genus $(a-1)(b-1)$ with $a \le b$ that admits an involution $\sigma$. Assume that either $a$ is odd or $b$ is odd. Let $d \coloneq a$ if $a$ is odd, otherwise let $d \coloneq b$. Assume moreover that  
\[
    g(C) > 2 g(C/\langle \sigma \rangle) + (d-1).
\]
    
Then $C$ cannot embed in $\P^1 \times \P^1$ with bidegree $(a,b)$.    
\end{proposition}

\begin{proof}
It follows from the Castelnuovo--Severi inequality, combined with the fact that $\gcd(d,2) = 1$ and the fact that $C \mapsto C/\langle \sigma \rangle$ is a Galois cover.
\end{proof}

The same argument can be applied with two distinct and commuting involutions:

\begin{proposition}
Let $C/\C$ be a curve of genus $(a-1)(b-1)$ with $a \le b$ that admits two distinct and commuting involutions $\sigma_1,\sigma_2$. Assume that either $a$ is odd or $b$ is odd. Let $d \coloneq a$ if $a$ is odd, otherwise let $d \coloneq b$. Assume moreover that  
\[
    g(C) > 4 g(C/\langle \sigma_1,\sigma_2 \rangle) + 3(d-1).
\]
    
Then $C$ cannot embed in $\P^1 \times \P^1$ with bidegree $(a,b)$.    
\end{proposition}

\begin{proof}
It follows from the Castelnuovo--Severi inequality, combined with the fact that $\gcd(d,2) = 1$ and the fact that $C \mapsto C/\langle \sigma_1, \sigma_2 \rangle$ is a Galois cover.
\end{proof}

The following is known as the Tower Theorem:
\begin{theorem}\cite[Theorem 4.4]{NO24},\cite[Proposition 2.4]{Poonen07}\label{thm: tower-theorem}
Let $C$ be a curve defined over a perfect field $k$ and $f : C \to \P^1$ be a non-constant morphism over $\overline{k}$ of degree $d$. Then there exists a curve $C'$ defined over $k$ and a non-constant morphism $C \to C'$ defined over $k$ of degree $d'$ dividing $d$ such that
\[
g(C') \le \left(\frac{d}{d'} -1 \right)^2.
\]
\end{theorem}

\begin{corollary}\cite[Corollary 4.5]{NO24}
\label{corollary: map_over_k}
Let $C$ be a curve defined over a perfect field $k$ such that $C(k) \ne \emptyset$ and let $f : C \to \P^1$ be a non-constant morphism over $\overline{k}$ of prime degree $d$ such that $g(C) > (d - 1)^2$. Then there exists a non-constant morphism $C \to \P^1$ of degree $d$ defined over $k$.
\end{corollary}

For trigonal and tetragonal curves we have the following results:
\begin{proposition}\cite[Corollary 4.6]{NO24} 
 \begin{enumerate}[(i)]
\item Let $C$ be a curve over $\Q$ of genus $\ge 5$  which is trigonal over $\C$ and such that $C(\Q)\ne \emptyset$. Then $C$ is trigonal over $\Q$.
\item Let $C$ be a curve defined over $\Q$ with $\gon_{\C}(C) = 4$ and $g(C) \ge 10$ and such that $C(\Q)\ne \emptyset$. Then $\gon_{\Q}(C) = 4$.
\end{enumerate}  
\end{proposition}

We can deduce an immediate corollary for low gonality curves:
\begin{corollary}\label{corollary: low_gon_eq}
Let $C$ be a curve over $\Q$ of genus $\ge 10$ such that $C(\Q)\ne \emptyset$ and $\gon_{\C}(C) \le 4$. Then
\[
\gon_{\Q}(C) = \gon_{\C}(C).
\]
\end{corollary}

\subsection{Geometry of curves embedded in $\P^1 \times \P^1$ or $\P^2$}

The gonality of curves embedded in the toric surface $\P^2$ or $\P^1 \times \P^1$ is well understood due to the following:

\begin{theorem}[Max Noether]
\label{thm: Max_Noether}
Let $K$ be a number field and $C \hookrightarrow \P^2_{\overline{K}}$ a smooth plane curve of degree $d \ge 3$. Then $\gon_{K}(C) = d-1$ if $C(K) \ne \emptyset$ and $\gon_{K}(C) = d$ if $C(K) = \emptyset$.

\end{theorem}

\begin{proposition}\cite[Theorem 1.2]{CDJP19}
Let $K$ be an algebraically closed field of characteristic zero and consider a smooth curve $C \subset \P^1 \times \P^1$ of bidegree $(a,b)$. Then 
\[
\gon_{\C}(C) = \min \{ a, b \}.
\]
\end{proposition}

We recall a theorem of Serrano \cite[Theorem 3.1]{Serrano87} as stated in \cite[Theorem 1.2]{Paoletti95}:

\begin{theorem}[Serrano]
\label{thm: Serrano}
Let $C$ be an irreducible smooth curve contained in a smooth surface $S$. Suppose that there exists a morphism $\Phi  : C \to \P^1$ of degree $d$. If $C^2 > (d+1)^2$, then there exists a morphism $\Psi : S \to \P^1$ extending $\Phi$.
\end{theorem}

We know that if $C \hookrightarrow \P^1 \times \P^1$ with bidegree $(a,b)$ then $C^2 = 2ab$.

\begin{theorem}\cite[Theorem 3]{Takahashi12}\label{thm: Takahashi_extension_theorem }

Let $C \sim aC_0 + bf$ be a nonsingular projective curve on the Hirzebruch surface $\mathcal{H}_e$ $(e \ge 0)$ and assume that $a, b \ge 3$ and $(e, a) \ne (1, b), (1, b - 1)$. Then, every automorphism of $C$ can be extended to an automorphism of $\mathcal{H}_e$. Namely, for every $\sigma \in \Aut(C)$, there exists $\hat{\sigma} \in \Aut(\mathcal{H}_e)$ such that $\hat{\sigma}(C)=C$ and $\hat{\sigma}|_C =\sigma$.

\end{theorem}

Note that in this notation $\mathcal{H}_0 = \P^1 \times \P^1$.

\subsection{Restrictions from fixed points}
The number of fixed points of an involution can only take three different values when the curve is geometrically trigonal:

\begin{lemma}\cite[Lemma 3.4]{Schweizer15}
\label{lemma: trigonal_fixed_involution}
Let $C$ be a trigonal curve of genus $g$ and $\sigma$ an involution on $C$.
\begin{enumerate}[(a)]
    \item If $g$ is odd, then $\sigma$ has exactly $4$ fixed points.
    \item If $g$ is even, then $\sigma$ has $2$ or $6$ fixed points.
\end{enumerate}
\end{lemma}

For curves that admit and embedding in either $\P^2$ or $\P^1 \times \P^1$, the number of points fixed by involutions is also restricted to few values:

\begin{theorem}{\cite[Remark 2.1 (i) and Theorem 2.2 with $n=2$]{HKKO10}} \label{theorem: smooth_plane_fixed_points}
Let $C$ be a smooth plane curve of degree $d \ge 4$ and $\sigma$ an involution of $C$. Then the involution $\sigma$ has $\#\Fix_C(\sigma) = d + \frac{1-(-1)^d}{2}$ fixed points.
\end{theorem}

\begin{lemma}
\label{lem:fixed-point-obstruction-P1xP1}
Let
\[
C\subset S\coloneq \P^1 \times \P^1
\]
be a smooth curve of bidegree $(a,b)$, where $3\leq a\leq b$, and
let $\iota \in \Aut(C)$ be a nontrivial involution. Set $r(\iota)\coloneq\#\Fix_C(\iota)$. Then $\iota$ extends to an involution
$\widetilde\iota\in\Aut(S)$ preserving $C$, and, up to conjugation by $G \coloneq \operatorname{PGL}_2\times\operatorname{PGL}_2$, the four possibilities for $\widetilde\iota$ and their consequences are
\\
\[
\begin{array}{c|c|c|c}
\text{representative}
& \Fix_S(\widetilde\iota)
& \text{condition forced on }(a,b)
& r(\iota)\\
\hline
(\tau,\tau)
& \text{four points}
& \text{none}
& \leq 4\\
(\tau,1)
& \text{a divisor of type }(2,0)
& a\text{ even}
& 2b\\
(1,\tau)
& \text{a divisor of type }(0,2)
& b\text{ even}
& 2a\\
s
& \text{a divisor of type }(1,1)
& a=b
& a+b=2a,
\end{array}
\]
\\
where $\tau$ is any nontrivial involution of $\mathbb P^1$ and $s$
interchanges the two factors. In the first row, Riemann--Hurwitz further
gives
\[
r(\iota)=
\begin{cases}
0\text{ or }4,&a,b\text{ both even},\\
2,&\text{otherwise}.
\end{cases}
\]
Consequently, if $a<b$, then
\[
r(\iota)\in
\begin{cases}
\{2\},
    & a,b\text{ both odd},\\[2mm]
\{2,2b\},
    & a\text{ even and }b\text{ odd},\\[2mm]
\{2,2a\},
    & a\text{ odd and }b\text{ even},\\[2mm]
\{0,4,2a,2b\},
    & a,b\text{ both even}.
\end{cases}
\]
If $a=b$, then
\[
r(\iota)\in
\begin{cases}
\{2,2a\},&a\text{ odd},\\[2mm]
\{0,4,2a\},&a\text{ even}.
\end{cases}
\]
\end{lemma}

\begin{proof}
By Takahashi's extension \cref{thm: Takahashi_extension_theorem }, the involution $\iota$ extends
to an automorphism $\widetilde\iota\in\Aut(S)$ preserving
$C$ and restricting to $\iota$.

The restriction homomorphism
\[
\Aut(S,C)\longrightarrow\Aut(C)
\]
is injective. Indeed, if $(\alpha,\beta)\in G$ restricts to the identity
on $C$, the surjectivity of the two projections $C\to\mathbb P^1$
implies $\alpha=\beta=1$. An automorphism interchanging the factors
cannot restrict to the identity on $C$, since then $C$ would be
contained in its fixed graph, which has bidegree $(1,1)$. It follows
that $\widetilde\iota^2=1$.

Recall that $\Aut(S)=G\rtimes\langle s\rangle$, where $\langle s\rangle \cong \Z/2\Z$ . If $\widetilde\iota=(\alpha,\beta)\in G$, then
$\alpha^2=\beta^2=1$. Every nontrivial involution of $\mathbb P^1$ is
conjugate to $\tau$, so according as both, only the first, or only the
second of $\alpha,\beta$ are nontrivial, $\widetilde\iota$ is conjugate
under $G$ to $(\tau,\tau)$, $(\tau,1)$, or $(1,\tau)$. If instead
$\widetilde\iota$ interchanges the factors, write
\[
\widetilde\iota=(\alpha,\beta)s.
\]
The relation $\widetilde\iota^2=1$ gives $\beta=\alpha^{-1}$, and
conjugation by an element of $G$ carries $\widetilde\iota$ to $s$.
This proves that there are exactly four conjugacy classes of nontrivial
involutions under $G$.

Their fixed loci are immediate. The fixed locus of $(\tau,\tau)$
consists of four points. The fixed locus of $(\tau,1)$ is the union of
the two fibers over the fixed points of $\tau$, hence is a divisor of
type $(2,0)$; similarly, the fixed locus of $(1,\tau)$ has type $(0,2)$.
Finally, the fixed locus of $s$ is the diagonal, of type $(1,1)$.

We next record the numerical conditions forced by the invariance of
$C$. For $(\tau,1)$, choose coordinates in which
\[
\tau([X_0:X_1])=[-X_0:X_1]
\]
and write an equation of $C$ as
\[
F(X_0,X_1,Y_0,Y_1)
=\sum_{i=0}^{a}X_0^iX_1^{a-i}F_i(Y_0,Y_1).
\]
Invariance gives
\[
F(-X_0,X_1,Y_0,Y_1)=\lambda F(X_0,X_1,Y_0,Y_1)
\]
for some $\lambda\in\{\pm1\}$. Since $C$ has no ruling as a component,
$F_0$ and $F_a$ are both nonzero. Comparing these two terms gives $\lambda=1=(-1)^a$, so $a$ is even. Interchanging the factors shows that $(1,\tau)$ forces
$b$ to be even. In the factor-interchanging case, invariance of the
class $(a,b)$ under the exchange of the two rulings forces $a=b$.
There is no analogous numerical restriction in the $(\tau,\tau)$ case.

For any of the three cases with a fixed divisor $D$, the intersection
of $C$ with $D$ is transverse. Indeed, at a fixed point of a nontrivial
involution of a smooth characteristic-zero curve, the action on the
tangent line is $-1$, while the tangent line to an ambient fixed divisor
is the $+1$-eigenspace. Thus $r(\iota)=C\cdot D$. Using $(a,b)\cdot(c,d)=ad+bc$, we obtain respectively
\[
(a,b)\cdot(2,0)=2b,\qquad
(a,b)\cdot(0,2)=2a,\qquad
(a,a)\cdot(1,1)=2a.
\]

It remains only to refine the four-point case. Let $g'$ be the genus of
$C/\langle\iota\rangle$. Since $g(C)=(a-1)(b-1)$, the
Riemann--Hurwitz formula gives
\[
r(\iota)=2g(C)+2-4g'.
\]
Hence $r(\iota)\equiv2g(C)+2\pmod 4$. Together with
$r(\iota)\leq4$, this yields $r(\iota)\in\{0,4\}$ when $a$ and $b$
are both even, and $r(\iota)=2$ otherwise. Combining the four rows gives
the stated lists.
\end{proof}

\cref{lem:fixed-point-obstruction-P1xP1} has the following immediate corollary:

\begin{lemma} \label{lemma: 2a_obstruction}
Consider a modular or Shimura curve $X_0^D(N)$ of genus $(a-1)(b-1)$ with $a \ne b$ and $\min\{a,b \} \ge 3$. For $ m \parallel DN$ denote $r_m \coloneq \#\Fix_{X_0^D(N)}(w_m)$. If there are distinct Hall divisors $m_1,m_2 \parallel DN$ with $r_{m_1} = r_{m_2} = 2a$, and 
\[
r_{m_1m_2/(m_1,m_2)^2} \ne 2a,
\]
then no smooth embedding into $\P^1 \times \P^1$ of bidegree $(a,b)$ exists. (One may also interchange $a$ and $b$.)
\end{lemma}

\begin{proof}
Suppose that $X_0^D(N)$ admits an embedding in
$S=\mathbb{P}^1\times\mathbb{P}^1$ of bidegree $(a,b)$,
and put $\sigma_i=w_{m_i}$ for $i=1,2$.
By \cref{lem:fixed-point-obstruction-P1xP1}, each $\sigma_i$
extends uniquely to an involution $\widetilde{\sigma}_i$
of $S$, where uniqueness follows from the injectivity
of restriction proved there.

Since $a\neq b$ and $\#\operatorname{Fix}(\sigma_i)=2a>4$,
the classification in that lemma forces
\[
    \widetilde{\sigma}_i=(1,\tau_i),
\]
with $\tau_i$ a nontrivial involution of $\mathbb{P}^1$.
The Atkin--Lehner involutions $\sigma_1,\sigma_2$ are
distinct and commute, so their product is a nontrivial
involution. Its unique extension is $\widetilde{\sigma}_1\widetilde{\sigma}_2
    =(1,\tau_1\tau_2)$.
Applying the same lemma therefore gives $\#\operatorname{Fix}(\sigma_1\sigma_2)=2a$.

But $\sigma_1\sigma_2 =w_{m_1m_2/\gcd(m_1,m_2)^2}$, contradicting the assumed fixed-point count.
\end{proof}

The following proposition will be used to rule out smooth embeddings in $\P^1 \times \P^1$ of bidegree $(a,a)$:
\begin{proposition}\label{prop:balanced-double-quotient}
Let $n\geq 4$ be even, and let $\pi\colon X\longrightarrow Y$ be a degree-two morphism of smooth projective connected curves over $\C$,
with
\[
  g(X)=(n-1)^2,
  \qquad
  g(Y)=\frac{(n-1)(n-2)}{2}.
\]
If $X$ admits an embedding in $\P^1\times\P^1$ of bidegree $(n,n)$,
then at least one of the following holds:
\begin{enumerate}[(a)]
  \item\label{case:balanced-low-degree}
  $Y$ admits a morphism to $\P^1$ of degree $n/2$;
  \item\label{case:balanced-plane}
  $Y$ admits a smooth plane model of degree $n$.
\end{enumerate}
\end{proposition}

\begin{proof}
Suppose such an embedding exists, and identify $X$ with its image in
$S=\P^1\times\P^1$. Let $\sigma$ be the deck involution of $\pi$.
Riemann--Hurwitz gives
\[
  \#\Fix_X(\sigma)
  =2g(X)+2-4g(Y)
  =2n>4.
\]
By \cref{lem:fixed-point-obstruction-P1xP1}, $\sigma$ extends to an involution $\widetilde{\sigma}$ of $S$.

Suppose first that $\widetilde{\sigma}$ preserves the two factors. Write
\[
  \widetilde{\sigma}=(\alpha,\beta).
\]
If both $\alpha$ and $\beta$ are nontrivial, the ambient fixed locus
consists of four points, contradicting the fixed-point count above.
They cannot both be trivial, since $\sigma\neq 1$. Thus exactly one
is trivial.

The corresponding ruling map $X\to\P^1$, of degree $n$, is
$\sigma$-invariant and therefore factors through $\pi$. The induced
map $Y\to\P^1$ has degree $n/2$, proving
(\ref{case:balanced-low-degree}).

Suppose instead that $\widetilde{\sigma}$ exchanges the factors.
After conjugation, it is the standard interchange. Its quotient map is
\[
  q\colon S\longrightarrow\Sym^2(\P^1)\simeq\P^2,
  \qquad
  q^*\mathcal O_{\P^2}(1)\simeq\mathcal O_S(1,1).
\]
Let $B=q(X)$. The restriction $q|_X$ has degree two, and $Y$ is the
normalization of $B$. By the projection formula,
\[
  2\deg B=(1,1)\cdot(n,n)=2n,
\]
so $\deg B=n$. Consequently,
\[
  p_a(B)=\frac{(n-1)(n-2)}{2}=g(Y).
\]
Since $Y$ is the normalization of the integral curve $B$, we obtain
\[
  \sum_{P\in B}\delta_P=p_a(B)-g(Y)=0.
\]
Thus $B$ is smooth and $Y\simeq B$, proving
(\ref{case:balanced-plane}).
\end{proof}

\subsection{Curves over finite fields}
A lower bound of the $\Q$-gonality can be obtained via finite field point counts: 
\begin{lemma}\cite[Lemma 3.5]{NO24} \label{NO-Fp}
Let $C/\Q$ be a curve, $p$ a prime of good reduction for $C$ and $q$ a power of $p$. Suppose $\#C(\F_q) > d(q + 1)$ for some $d$. Then $\gon_{\Q}(C) > d$.
\end{lemma}

\begin{lemma}[Ogg] \label{lemma: Ogg-Fp}
For a prime $p \nmid N$, let
\[
L_p(N) \coloneq \frac{p-1}{12}\psi(N) + 2^{\omega(N)},
\]
where
\[
\psi(N) = N \prod_{q \mid N}\left(1+\frac{1}{q}\right)
\]
and $\omega(N)$ is the number of distinct prime divisors of $N$. Then
\[
L_p(N) \leq \#X_0(N)(\mathbb{F}_{p^2}).
\]
\end{lemma}

\begin{remark}\label{remark : prime_gonality}
If $C/\Q$ is a curve with $C(\Q) \ne \emptyset$, $g(C) > (p-1)^2$, and $\gon_{\C}(C) = p$ is prime, then the Tower \cref{thm: tower-theorem} implies that $\gon_{\Q}(C) = p$.
\end{remark}

\begin{lemma}[Good reduction lemma, see {\cite[Lemma~1.8]{HS99a}}] \label{lemma : good_reduction_lemma}
Let $C_1$ and $C_2$ be smooth projective geometrically irreducible
curves defined over $\Q$, and let
\[
    f:C_1\longrightarrow C_2
\]
be a finite morphism of degree $d$ defined over $\Q$.
Assume that $C_1$ has good reduction at a prime $p$.

\begin{enumerate}
    \item[(i)]
    If $g(C_2)\geq 1$, then $C_2$ has good reduction at $p$, and
    $f$ induces a finite morphism
    \[
        \widetilde f:
        \widetilde C_1\longrightarrow \widetilde C_2
    \]
    of degree $d$ over $\F_p$, where $\widetilde C_i$
    denotes the reduction of $C_i$ at $p$.

    \item[(ii)]
    If $g(C_2)=0$, then there exists a finite morphism
    \[
        \widetilde f^{\,\prime}:
        \widetilde C_1\longrightarrow \widetilde C_2^{\,\prime}
    \]
    of degree $d'\leq d$ over $\F_p$, where
    $\widetilde C_2^{\,\prime}$ is a smooth rational curve defined
    over $\F_p$.
\end{enumerate}
\end{lemma}

The following lemma establishes the case $\gon_{\C}(X_0^D(N)) \le 2$, so we can later focus our attention to the case $\gon_{\C}(X_0^D(N)) \ge 3$:

\begin{lemma}\label{lem:low-gonality-embedding}
Let $C$ be a smooth projective connected curve over $\mathbb C$ of
genus $g$. If $\operatorname{gon}_{\mathbb C}(C)\leq 2$, then $C$ admits
a closed immersion into $\mathbb P^1\times\mathbb P^1$.
More precisely, one may choose the image to have bidegree $(1,1)$
if $g=0$, and bidegree $(2,g+1)$ if $g\geq 1$.
\end{lemma}

\begin{proof}
If $g=0$, then $C\simeq\mathbb P^1$, and the diagonal gives the
required embedding. Suppose that $g\geq 1$, and choose a morphism
$f\colon C\to\mathbb P^1$ of degree two.

We first choose a nonspecial, base-point-free line bundle $L$ of
degree $g+1$ that is not pulled back from $\mathbb P^1$ under $f$.
Such a choice is possible. Indeed, if $g=1$, every line bundle of
degree two is nonspecial and base-point-free. If $g\geq 2$, the
special locus in $\operatorname{Pic}^{g+1}(C)$ is empty for $g=2$
and has dimension at most $g-3$ otherwise, since
$\omega_C\otimes L^{-1}$ must be effective of degree $g-3$.
Moreover, if a nonspecial line bundle $L$ of degree $g+1$ has a
base point $P$, then Riemann--Roch and Serre duality give
\[
  L\simeq\omega_C\otimes\mathcal O_C(P-E)
\]
for some effective divisor $E$ of degree $g-2$. Such line bundles
therefore lie in a locus of dimension at most $g-1$.
Thus a general line bundle of degree $g+1$ is nonspecial and
base-point-free. Finally, the condition $L\simeq f^*M$ excludes
at most one point of $\operatorname{Pic}^{g+1}(C)$: it is possible
only when $g+1$ is even, in which case
$M\simeq\mathcal O_{\mathbb P^1}((g+1)/2)$.
Since $\dim\operatorname{Pic}^{g+1}(C)=g\geq 1$, the required
choice of $L$ exists.

By Riemann--Roch, $h^0(C,L)=2$, so the complete linear system
$|L|$ defines a morphism $h\colon C\to\mathbb P^1$ of degree
$g+1$. Consider
\[
  \varphi=(h,f)\colon C\longrightarrow\mathbb P^1\times\mathbb P^1,
\]
and let $\Gamma$ be its image with the reduced structure.
The degree of $\varphi$ onto $\Gamma$ divides $\deg(f)=2$.
If it were two, the second projection would induce a degree-one
morphism from the normalization of $\Gamma$ to $\mathbb P^1$.
Consequently, $h$ would factor through $f$, which would make
$L=h^*\mathcal O_{\mathbb P^1}(1)$ a pullback under $f$, contrary
to its choice. Hence $\varphi$ is birational onto $\Gamma$.

It follows that $\Gamma$ has bidegree $(2,g+1)$, and therefore
\[
  p_a(\Gamma)=(2-1)((g+1)-1)=g.
\]
Since $C$ is the normalization of $\Gamma$, equality of its genus
with $p_a(\Gamma)$ implies that $\Gamma$ is smooth. Thus
$\varphi\colon C\to\Gamma$ is an isomorphism, and $\varphi$ is
the required closed immersion.
\end{proof}

\section{Modular Curves}

Let $X \subset \P^1 \times \P^1$ be a modular curve of bidegree $(a,b)$. Without loss of generality assume that $2 \le a \le b$. Then 
\[
a = \gon_{\C} X  \ge \frac{1}{2} \cdot \frac{975((a-1)(b-1)-1)}{4096},
\]

and hence

\[
2 \le a \le b \le \frac{9167a}{975 (a - 1)} \le 18.
\]

If $a = 2$ then $X_0(N)$ is either elliptic or hyperelliptic. For $a \ge 3$, we need to check the following bidegrees:

\begin{itemize}
    \item $a = 3, \quad 3 \le b \le 14$,
    \item $a = 4, \quad 4 \le b \le 12$,
    \item $a \in \{ 5,6 \}, \quad a \le b \le 11$,
    \item $a \in \{ 7,8,9,10 \}, \quad a \le b \le 10$.
\end{itemize}

\begin{theorem}\cite[Theorem 3.3]{HS99a}
The modular curve $X_0(N)$ is a non-sub-hyperelliptic trigonal curve if and only if

\begin{align*}
  &\begin{aligned}
    \mathllap{N = 34,43,45,64} & \qquad (g = 3),
  \end{aligned} \\
  &\begin{aligned}
    \mathllap{N = 38, 44, 53, 54, 61, 81} & \qquad (g = 4).
  \end{aligned}
\end{align*}
    
\end{theorem}

Let $C$ be a smooth non-hyperelliptic curve of genus $4$ over an
algebraically closed field of characteristic zero. The canonical model $C \hookrightarrow \mathbf P^3$ is the complete intersection of a unique quadric $Q$ and a cubic
surface. Every $g^1_3$ on $C$ is cut out by a ruling of $Q$. Consequently,
if $Q$ is nonsingular, its two rulings induce two distinct $g^1_3$'s on
$C$; if $Q$ is a quadric cone, its unique ruling induces the unique
$g^1_3$ on $C$
\cite[Chapter~III, \S3]{ACGH85}.

\begin{proposition} For $N \in \{38,44,53,54,61 \}$, the modular curve $X_0(N)$ embeds in $\P^1 \times \P^1$ with bidegree $(3,3)$, while
$X_0(81)$ does not embed in $\P^1 \times \P^1$ with bidegree $(3,3)$.

\end{proposition}

\begin{proof}
The canonical models for $X_0(N)$ are taken from \cite{Galbraith96}. 

\begin{table}[htbp]
\centering
\small
\renewcommand{\arraystretch}{1.7}
\setlength{\tabcolsep}{8pt}
\begin{tabular}{c|p{0.82\textwidth}}
\hline
$N$ & Defining equations of $X_0(N)\subset\mathbf P^3_{[w:x:y:z]}$ \\
\hline

$38$
&
\[
\begin{aligned}
Q_{38}:&\quad
w^2+16wx+10x^2-18y^2-9z^2=0,\\
F_{38}:&\quad
2w^3-4w^2x-5x^3+4wx^2-7xy^2
+4xz^2+4wy^2+2wz^2=0.
\end{aligned}
\]
\\
\hline

$44$
&
\[
\begin{aligned}
Q_{44}:&\quad
3w^2+16wx+32x^2-4y^2+z^2=0,\\
F_{44}:&\quad
w^2x+8x^3+2wy^2-4xy^2-2wz^2-5xz^2=0.
\end{aligned}
\]
\\
\hline

$53$
&
\[
\begin{aligned}
Q_{53}:&\quad
x^2-w^2+2xy+2xz-11y^2-10yz-7z^2=0,\\
F_{53}:&\quad
x^2z+xy^2+xyz+5xz^2+2y^2z+yz^2+6z^3=0.
\end{aligned}
\]
\\
\hline

$54$
&
\[
\begin{aligned}
Q_{54}:&\quad
w^2+2x^2-2y^2-z^2=0,\\
F_{54}:&\quad
w^3+3wz^2-x^3-3xy^2=0.
\end{aligned}
\]
\\
\hline

$61$
&
\[
\begin{aligned}
Q_{61}:&\quad
w^2-x^2+2xy-6xz+3y^2+6yz-5z^2=0,\\
F_{61}:&\quad
x^2z+xy^2+xyz+5xz^2+4y^2z+5yz^2+6z^3=0,
\end{aligned}
\]
\\
\hline

$81$
&
\[
\begin{aligned}
Q_{81}:&\quad
w^2-x^2+12z^2=0,\\
F_{81}:&\quad
w^2x+3xz^2+9z^3-y^3=0.
\end{aligned}
\]
\\
\hline

\end{tabular}
\caption{Galbraith's canonical models of the modular curves $X_0(N)$. Each curve is
the complete intersection
$X_0(N)=V(Q_N,F_N)\subset\mathbf P^3_{[w:x:y:z]}$.}
\label{tab:canonical-models-X0N}
\end{table}

For $N \in \{38,44,53,54,61 \}$ we get that the quadric is nonsingular over $\C$, hence $X_0(N)$ embeds in $\P^1 \times \P^1$ with bidegree $(3,3)$.

For $X_0(81)$, since the symmetric matrix of the quadric has rank $3$ it follows that $X_0(81)$ has a unique $g_3^1$. So it cannot embed in $\P^1 \times \P^1$ with bidegree $(3,3)$.

\end{proof}

\renewcommand{\arraystretch}{1.7}
\setlength{\tabcolsep}{8pt}

\begin{longtable}{c|p{0.82\textwidth}}
\hline
$N$ & Defining equation of the $(3,3)$-model in
$\mathbf P^1\times\mathbf P^1$ \\
\hline

\endfirsthead

\caption{$(3,3)$-models of the modular curves $X_0(N)$ in
$\mathbf P^1\times\mathbf P^1$.}
\label{tab: P1twice_models}\\
\endlastfoot

$38$
&
\[
\begin{aligned}
0={}&
u_0^3\Bigl[
(588+980s)v_0^2v_1
+(42-518s)v_0v_1^2
+(-366+398s)v_1^3
\Bigr]\\
&+u_0^2u_1\Bigl[
(-588+980s)v_0^3
+3528\,v_0^2v_1
-(3150+1428s)v_0v_1^2
+(267+1019s)v_1^3
\Bigr]\\
&+u_0u_1^2\Bigl[
(42+518s)v_0^3
+(3150-1428s)v_0^2v_1
-4284\,v_0v_1^2
+(1059+533s)v_1^3
\Bigr]\\
&+u_1^3\Bigl[
(366+398s)v_0^3
+(267-1019s)v_0^2v_1
+(-1059+533s)v_0v_1^2
+372\,v_1^3
\Bigr].
\end{aligned}
\]
\\
\hline

$44$
&
\[
\begin{aligned}
0={}&
u_0^3\Bigl[
(8s-13)v_0^3
-(8s+8)v_0v_1^2
\Bigr]\\
&+u_0^2u_1\Bigl[
(72s-9)v_0^2v_1
+(8-8s)v_1^3
\Bigr]\\
&+u_0u_1^2\Bigl[
-(72s+72)v_0^3
+(72s+9)v_0v_1^2
\Bigr]\\
&+u_1^3\Bigl[
(72-72s)v_0^2v_1
+(8s+13)v_1^3
\Bigr].
\end{aligned}
\]
\\
\hline

$53$
&
\[
\begin{aligned}
0={}&
u_0^3\Bigl[
54\,v_0^2v_1
+234\,v_0v_1^2
+(279-9s)v_1^3
\Bigr]\\
&+u_0^2u_1\Bigl[
108\,v_0^3
+1062\,v_0^2v_1
+(3294+12s)v_0v_1^2
+(2955+51s)v_1^3
\Bigr]\\
&+u_0u_1^2\Bigl[
720\,v_0^3
+(5706-24s)v_0^2v_1
+(13206-24s)v_0v_1^2
+(10029+269s)v_1^3
\Bigr]\\
&+u_1^3\Bigl[
(1404+72s)v_0^3
+(8346-72s)v_0^2v_1
+(16818-676s)v_0v_1^2
+(7873-311s)v_1^3
\Bigr].
\end{aligned}
\]
\\
\hline

$54$
&
\[
\begin{aligned}
0={}&
u_0^3\Bigl[
8v_0^3-8v_1^3
\Bigr]
+u_1^3\Bigl[
v_0^3+8v_1^3
\Bigr].
\end{aligned}
\]
\\
\hline

$61$
&
\[
\begin{aligned}
0={}&
u_0^3\Bigl[
1472\,v_0^3
+(4688+1776s)v_0^2v_1
+(4324+3868s)v_0v_1^2
+(1029+2028s)v_1^3
\Bigr]\\
&+u_0^2u_1\Bigl[
(-4688+1776s)v_0^3
-16680\,v_0^2v_1
-(17635+6720s)v_0v_1^2
-(5405+4835s)v_1^3
\Bigr]\\
&+u_0u_1^2\Bigl[
(4324-3868s)v_0^3
+(17635-6720s)v_0^2v_1
+20850\,v_0v_1^2
+(7325+2775s)v_1^3
\Bigr]\\
&+u_1^3\Bigl[
(-1029+2028s)v_0^3
+(-5405+4835s)v_0^2v_1
+(-7325+2775s)v_0v_1^2
-2875\,v_1^3
\Bigr].
\end{aligned}
\]
\\
\hline

\end{longtable}

\noindent
Here $([u_0:u_1],[v_0:v_1])$ are homogeneous coordinates on $\P^1\times\P^1$ and
$s=\sqrt{-3},\sqrt{-2},\sqrt{-15},\ \text{---},\ \sqrt{-1}$ for $N = 38,44,53,54,61$, respectively.

Considering the genus and \cite[Theorem 1.1]{NO24}, the only candidates of bidegree $(4,b)$ are the genus $9$ curves $X_0(N)$ with 
\[
N \in \{ 66, 70, 87, 88, 95, 96, 107 \}.\] For each level, we use Magma \cite{magma} to compute the Betti number $\beta_{2,4} = 5 = 9 - 4$, which guarantees that $g_4^{1}$ exists and is unique, see \cite[Theorem 4.1 and Theorem 4.4]{Schreyer91}. Hence $X_0(N)$ does not embed in $\P^1 \times \P^1$ with bidegree $(4,4)$ for any of these levels.

After applying all filters deriving from the Castelnuovo--Severi inequality and \cref{lem:fixed-point-obstruction-P1xP1}, we are left with the $19$ remaining levels $N$:

\[
\begin{array}{*{10}{r}}
212, & 216, & 237, & 307, & 304, & 433, & 529, & 332, & 417, & 423, \\
547, & 653, & 384, & 512, & 514, & 769, & 586, & 664, & 977.
\end{array}
\]

\begin{remark}
Notice that $N = 332$ could have $(a,b) = (5,11)$ or $(6,9)$, but the filtering only excludes $(a,b) = (5,11)$.
\end{remark}

\begin{proposition}\label{prop:gonality-at-least-seven}
For every
\[
    N\in\{304,417,433,529,547\},
\]
one has $\gon_{\C} X_0(N)\geq 7$.
In particular, none of these curves admits a geometric
embedding in $\P^1\times\P^1$ of
bidegree $(6,b)$.
\end{proposition}

\begin{proof}
Fix one of the listed levels and put $X \coloneq X_0(N)$.
Choose $q$ as in the following table, and let
$p=\text{char}(\F_q)$.
In each case $p\nmid N$, so $X$ has good reduction at $p$;
write $\widetilde{X}$ for this reduction.
We use the following genera and point counts:
\[
\begin{array}{c|c|c|r|r|r}
N & g(X) & q
  & \#\widetilde{X}(\mathbb{F}_q)
  & 6(q+1)
  & 2(q+1+8\sqrt{q}) \\ \hline
304 & 35 & 9 & 92 & 60 & 68 \\
417 & 45 & 4 & 56 & 30 & 42 \\
433 & 35 & 9 & 88 & 60 & 68 \\
529 & 35 & 4 & 51 & 30 & 42 \\
547 & 45 & 4 & 54 & 30 & 42
\end{array}
\]
Since $\#\widetilde{X}(\mathbb{F}_q)>6(q+1)$,
\cref{NO-Fp} gives
\begin{equation}\label{eq:no-rational-map-degree-six}
    \operatorname{gon}_{\mathbb{Q}}(X)>6.
\end{equation}

On the other hand,
$\operatorname{gon}_{\mathbb{C}}(X)\geq 6$
by \cite[Proposition~5.23]{NO24}, since $N\geq 198$.
Suppose, for a contradiction, that
$\operatorname{gon}_{\mathbb{C}}(X)=6$.
The Tower~\cref{thm: tower-theorem} gives a smooth
projective geometrically integral curve $Y/\mathbb{Q}$
and a morphism  $\pi:X\longrightarrow Y$ defined over $\Q$, of degree $d\mid 6$, such that
\[
    g(Y)\leq \left(\frac{6}{d}-1\right)^2.
\]
The image of a rational cusp of $X$ gives
$Y(\mathbb{Q})\neq\varnothing$.

If $d=1$, then $X\simeq Y$, giving
$g(X)\leq 25$, contrary to the table.
If $d=6$, then $g(Y)=0$ and the rational point gives
$Y\simeq\mathbb{P}^1_{\mathbb{Q}}$; thus $\pi$
contradicts \eqref{eq:no-rational-map-degree-six}.
If $d=3$, then $g(Y)\leq 1$.
Since $Y$ has a rational point, Riemann--Roch gives
a morphism $Y\to\mathbb{P}^1$ over $\mathbb{Q}$
of degree at most two.
Composing with $\pi$ again contradicts
\eqref{eq:no-rational-map-degree-six}.

It remains to consider $d=2$, for which $g(Y)\leq 4$.
The case $g(Y)=0$ again contradicts
\eqref{eq:no-rational-map-degree-six}.
We may therefore assume $1\leq g(Y)\leq 4$.
By \cref{lemma : good_reduction_lemma}, $Y$ has good
reduction at $p$, and $\pi$ induces a finite morphism $\widetilde{\pi}:
    \widetilde{X}\longrightarrow\widetilde{Y}$
of degree two over $\mathbb{F}_p$.
Consequently, the Hasse--Weil bound gives
\[
\begin{aligned}
    \#\widetilde{X}(\mathbb{F}_q)
    &\leq 2\#\widetilde{Y}(\mathbb{F}_q) \\
    &\leq 2\bigl(q+1+2g(Y)\sqrt{q}\bigr) \\
    &\leq 2\bigl(q+1+8\sqrt{q}\bigr).
\end{aligned}
\]
This contradicts the last column of the table in
every case. Hence
$\operatorname{gon}_{\mathbb{C}}(X)\geq 7$.

Finally, an embedding of bidegree $(6,b)$ would
supply a ruling map of degree six, contradicting
this gonality bound.
\end{proof}

\begin{proposition}
$X_0(769)$ does not embed in $\P^1 \times \P^1$ with bidegree $(8,10)$.
\end{proposition}

\begin{proof}
The quotient $X_0(769)^+$ has genus $27$, hence by the Castelnuovo--Severi inequality it follows that any $8$-gonal map would factor through $X_0(769)^+$, which is not tetragonal by \cite[Theorem 1.2]{Orlic25b}.
\end{proof}

\begin{proposition}

We have the following:

\begin{itemize}
    \item $X_0(332)$ does not geometrically embed in $\P^1 \times \P^1$ with bidegree $(6,9)$,
    \item $X_0(423)$ does not geometrically embed in $\P^1 \times \P^1$ with bidegree $(6,10)$,
    \item $X_0(514)$ does not geometrically embed in $\P^1 \times \P^1$ with bidegree $(8,10)$,
    \item $X_0(586)$ does not geometrically embed in $\P^1 \times \P^1$ with bidegree $(9,10)$.
\end{itemize}

\end{proposition}

\begin{proof}
We use \cref{lemma: 2a_obstruction} together with the following data:
\[
\begin{array}{c|ccc}
N & r_{m_1} & r_{m_2} & r_{m_3}\\
\hline
332 & r_{4}=2 & r_{83}=18 & r_{332}=18,\\
423 & r_{9}=0 & r_{47}=20 & r_{423}=20,\\
514 & r_{2}=4 & r_{257}=16 & r_{514}=16,\\
586 & r_{2}=2 & r_{293}=18 & r_{586}= 18.
\end{array}
\]

\end{proof}

\begin{proposition}
$X_0(653)$ does not geometrically embed in $\P^1 \times \P^1$ with bidegree $(7,10)$.
\end{proposition}

\begin{proof}
Assume for a contradiction that $X_0(653)$ geometrically embeds in $\P^1 \times \P^1$ with bidegree $(7,10)$. Then by \cref{remark : prime_gonality}
$\gon_{\C}(X_0(653)) = \gon_{\Q}(X_0(653)) = 7$. But by \cref{lemma: Ogg-Fp} we know that $(X_0(653))(\F_4) \ge 57$, hence 
\[
\gon_{\Q}(X_0(653)) \ge \gon_{\F_4}(X_0(653)) \ge 12.
\]
\end{proof}

\begin{proposition}\label{prop: not_plane_sextic}
For $N \in \{ 212, 216, 237, 307 \}$, we have that $X_0(N)^+$ is a genus $10$ curve with $\gon_{\C}(X_0(N)^+) \ge 5$. Moreover, $X_0(N)^+$ is not a smooth plane sextic. 
\end{proposition}

\begin{proof}
A smooth plane sextic has the Betti number $\beta_{2,4} = 27$, while the curves above have $\beta_{2,4} = 0$. For $N \in \{ 212, 216, 237\}$, one can also use a further quotient and \cref{theorem: smooth_plane_fixed_points} to obtain a contradiction. Notice that since $\beta_{2,4} = 0$, there is no $g_4^1$ and since these curves are not trigonal, it follows that $\gon_{\C}(X_0(N)^+) \ge 5$.
\end{proof}

\begin{proposition}
For $N\in\{212,216,237,307\}$, the genus-$25$ curve $X_0(N)$ does not
geometrically embed in $\P^1\times\P^1$ with bidegree $(6,6)$.
\end{proposition}

\begin{proof}
Put
\[
  X \coloneq X_0(N),\qquad Y \coloneq X_0(N)^+.
\]
The quotient by the Fricke involution $w_N$ has degree two, with
$g(X)=25$ and $g(Y)=10$. \cref{prop:balanced-double-quotient},
applied after base change to $\C$ with $n=6$, shows that a hypothetical
$(6,6)$-embedding would give either a degree-three map $Y\to\P^1$
or a smooth plane sextic model of $Y$. \cref{prop: not_plane_sextic} excludes
both possibilities.
\end{proof}

\begin{proposition}
For $N\in\{384,512\}$, the modular curve $X_0(N)$ does not geometrically
embed in $\P^1\times\P^1$ with bidegree $(8,8)$.
\end{proposition}

\begin{proof}
Consider the degree-two degeneracy map $X_0(N)\longrightarrow X_0(N/2)$.
The source and target have genera $49$ and $21$, respectively. \cref{prop:balanced-double-quotient}, applied after base change
to $\C$ with $n=8$, shows that a hypothetical $(8,8)$-embedding would
imply that $X_0(N/2)$ admits either a geometric degree-four map to
$\P^1$ or a geometric smooth plane octic model.

The first alternative is excluded by $\gon_{\C}X_0(N/2)\geq 6$,
as given by \cite[Proposition~5.23]{NO24}. The second is excluded by
\cite[\S5.6]{AALG23}, since $N/2\in\{192,256\}$.
\end{proof}

\begin{proposition}
For $N\in\{664,977\}$, the modular curve $X_0(N)$ does not geometrically
embed in $\P^1\times\P^1$ with bidegree $(10,10)$.
\end{proposition}

\begin{proof}
Put
\[
  X \coloneq X_0(N),\qquad Y \coloneq X_0(N)^+.
\]
The Fricke quotient $X\to Y$ has degree two, with $g(X)=81$ and
$g(Y)=36$. Moreover, $Y(\Q)\neq\varnothing$, since the image of a
rational cusp is rational.

Suppose that $X$ geometrically embeds in $\P^1\times\P^1$ with
bidegree $(10,10)$. \cref{prop:balanced-double-quotient}
gives either a geometric degree-five map $Y\to\P^1$ or a geometric
smooth plane model of $Y$ of degree ten.

In the first case, \cref{corollary: map_over_k} applies because $5$ is prime and $g(Y)=36>(5-1)^2$. It yields a degree-five map to $\P^1$ defined over $\Q$.

In the second case, the smooth plane model descends to $\Q$ by
\cite[Remark~20]{RX18}: its degree is greater than three and is not
divisible by three. Projection from the image of a rational point of
$Y$ then gives a degree-nine map to $\P^1$ defined over $\Q$.
Thus either alternative implies $\gon_{\Q}(Y)\leq 9$.

The following point counts contradict this bound:
\begin{center}
\begin{tabular}{rrrr}
\toprule
$N$ & $q$ & $\#Y(\F_q)$ & $9(q+1)$ \\
\midrule
$664$ & $9$ & $102$ & $90$ \\
$977$ & $4$ & $50$ & $45$ \\
\bottomrule
\end{tabular}
\end{center}
Indeed, in each row the characteristic of $\F_q$ does not divide
$N$, so $X$ has good reduction there. Since $g(Y)>0$, \cref{lemma : good_reduction_lemma} gives
good reduction for $Y$ as well. \cref{NO-Fp} now implies
$\gon_{\Q}(Y)>9$, a contradiction.
\end{proof}

\section{Shimura Curves $X_0^D(N)$}

Using the Abramovich--Kim--Sarnak bound, we are left to analyze

\[
2 \le a \le b \le \frac{9167a}{975(a-1)},
\]

which yields the following bidegrees:

\begin{itemize}
    \item $(2,b)$, 
    \item $(3,b)$,
    \item $a=4, \quad 4 \le b \le 12$,
    \item $a \in \{5,6 \}, \quad a \le b \le 11$,
    \item $a \in \{ 7,8,9,10 \}, \quad a \le b \le 10$.
\end{itemize}

Note that $g \le 81$.

A lower bound for the genus is given by the following: 
\begin{lemma}{\cite[Lemma 10.6]{Saia24}}
\label{lemma: genus_upper_bound}
For $D>1$ an indefinite rational quaternion discriminant and $N \in \mathbb{Z}_{\ge 1}$ relatively prime to $D$, we have
\[ g(X_0^D(N)) > 1 + \frac{DN}{12}\left( \frac{1}{e^\gamma \log\log(DN) + \frac{3}{\log\log{6}}} \right)- \frac{7\sqrt{DN}}{3},\]
where $\gamma$ is the Euler constant.
\end{lemma}


If $DN \ge 87,641$ then $g(X_0^D(N)) > 81$. Thus we only need to check  $DN \le 87,640$, for which $X_0^D(N)$ is not geometrically hyperelliptic.

There are $457$ such candidate pairs $(D,N)$.

If $a = 2$ and $b \ge 3$ then $X_0^D(N)$ is geometrically hyperelliptic. All such curves were classified in \cite{GY17}. They embed in $\P^1 \times \P^1$ with bidegree $(2,g+1)$.

If $a=3$ then $X_0^D(N)$ is geometrically trigonal, and they were classified in \cite[Theorem 7.4]{PS25a}. There are three genus $2$ curves, and the genus $4$ curves $X_0^{106}(1)$ and $X_0^{118}(1)$.

Many of the candidates with bidegree $(4,b)$ are excluded using \cite[Table 13]{MPSS25}.

\begin{proposition} The Shimura curves $X_0^6(35)$ and $X_0^{10}(21)$ admit a smooth model in  $\P^1 \times \P^1$ of bidegree $(4,4)$.
\end{proposition}

\begin{proof}
For $(D,N) = (6,35)$ use $W_1 \coloneq \langle w_6, w_{70} \rangle$ and $W_2 \coloneq \langle w_{10},w_{21} \rangle$, while for $(D,N) = (10,21)$ use $W_1 \coloneq \langle w_3, w_{35} \rangle$ and $W_2 \coloneq \langle w_{5},w_{42} \rangle$. In both cases $g(X_0^D(N)/W_1) = g(X_0^D(N)/W_2) = 0$ and $W_1 \cap W_2 = \{\id \}$. The quotient maps $X \to X/W_i \simeq \P^1$ have degree $4$ for $i \in \{ 1 ,2 \}$ and since the two subgroups intersect trivially, the two quotient maps generate the full function field, so their product is birational onto its image. Therefore the image has bidegree $(4,4)$ and since $g(X_0^D(N)) = (4-1)^2$, the birational image has arithmetic genus equal to the genus of $X_0^D(N)$, hence it is smooth.
\end{proof}

\begin{proposition} \label{prop: 14-17)}
The Shimura curve $X_0^{14}(17)$ admits a smooth model in  $\P^1 \times \P^1$ of bidegree $(4,4)$.
\end{proposition}

\begin{proof}
Let $X \coloneq X_0^{14}(17)$. 
Consider the quotients 
\[
Y \coloneq X/\langle w_2 \rangle, \quad Z \coloneq  X/\langle w_{238} \rangle,
\]

which are both genus $3$ curves admitting a degree $2$ map to genus $2$ curves, hence hyperelliptic by \cite[Theorem 3.4]{Pol06}. So we have two degree $4$ maps
\[
\pi_1 : X \to \P^1, \quad \pi_2 : X \to \P^1.
\]

Define a morphism $\Psi = (\pi_1,\pi_2) : X \to \P^1 \times \P^1$.

It remains to check that $\Psi$ is birational onto its image. Let the degree of $\Psi$ onto its image be $e$.
Since both coordinate projections have degree $4$ on $X$, the image has bidegree

\[
\left( \frac{4}{e},\frac{4}{e} \right),
\]

so $e|4$.

If $e = 4$, then $\pi_1$ and $\pi_2$ generate the same degree 4 pencil. The common map is invariant under both $w_2$ and $w_{238}$, hence it descends through 
\[
X \to X/\langle w_2, w_{238} \rangle.
\]

But $g(X/\langle w_2, w_{238} \rangle) = 1 \ne 0$, so the common degree $4$ map to $\P^1$ cannot arise this way. Thus $e \ne 4$.

If $e = 2$, then the image has bidegree $(2,2)$. The normalization of a curve of bidegree $(2,2)$ in $\P^1 \times \P^1$ has genus at most one. Therefore $X$ would admit a degree $2$ map to a curve of genus $0$ or $1$, making $X$ hyperelliptic or bielliptic. But Guo--Yang \cite{GY17} and Padurariu--Saia \cite{PS25a} exclude $X_0^{14}(17)$ from both possibilities. Thus $e \ne 2$.

Therefore $e=1$ and $\Psi$ is birational onto its image. Since each projection has degree $4$, the image has bidegree $(4,4)$.

Since $g(X) = 9 = (4-1)(4-1)$, while a curve of bidegree $(a,b)$ in $\P^1 \times \P^1$ has arithmetic genus $(a-1)(b-1)$, it follows that the arithmetic genus of the image equals the genus of $X$, so the birational image cannot have singularities. Hence $\Psi$ identifies $X$ with a smooth bidegree $(4,4)$ curve in $\P^1 \times \P^1$.

\end{proof}

After applying all filters deriving from the Castelnuovo--Severi inequality and \cref{lem:fixed-point-obstruction-P1xP1}, we are left with the $29$ remaining pairs $(D,N)$:

\[
\begin{aligned}
&(34, 7),\ (38, 5),\ (133, 1),\ (145, 1),\ (177, 1),\ (226, 1),\ (62, 5),\ (94, 3),\\
&(217, 1),\ (267, 1),\ (382, 1),\ (51, 7),\ (301, 1),\ (394, 1),\ (694, 1),\ (6, 157),\\
&(622, 1),\ (6, 271),\ (6, 277),\ (38, 29),\ (58, 19),\ (21, 32),\ (38, 31),\ (39, 23),\\
&(745, 1),\ (1365, 1),\ (889, 1),\ (955, 1),\ (1149, 1).
\end{aligned}
\]

\begin{proposition} 
 The curve $X_0^D(N)$ does not geometrically embed in $\P^1 \times \P^1$ with bidegree $(6,10)$ for 

\[
(D,N) \in \{ (6,271),(6,277),(38,29),(58,19) \}.
\]

\end{proposition}

\begin{proof}
Assume for a contradiction that $X_0^6(271)$ is $6$-gonal. Since the genus of $X_0^6(271)/\langle w_{1626} \rangle$ is $18$, it follows from the Castelnuovo--Severi inequality that the $6$-gonal map $X_0^6(271) \to \P^1$ factors through $X_0^6(271)/\langle w_{1626} \rangle$, hence $X_0^6(271)/\langle w_{1626} \rangle$ must be geometrically trigonal. Since $X_0^6(271)/\langle w_{1626} \rangle$ has a CM point defined over $\Q$, it follows that $X_0^6(271)/\langle w_{1626} \rangle$ must be trigonal over $\Q$. But
\[
\#(X_0^6(271)/\langle w_{1626} \rangle)(\F_{25}) = 114 > 3(25+1) = 78,
\]
contradiction.
A similar argument works for the remaining pairs, using the data from the table below. The number of $\F_p$-points was computed using the code that accompanies \cite{MPSS25}. To check the existence of a $\Q$-rational point coming from the image of a CM point we use the code from \cite{PS25b}.

{
\begin{longtable}{|c|c|c|c|c|}   \hline 
$(D,N)$ & $m$ & $g(X_0^D(N)/\langle w_{m} \rangle)$ & $\F_q$ & $\#(X_0^D(N)/\langle w_{m} \rangle)(\F_q)$  \\ \hline \hline
$(6,277)$ & $6 \cdot 277$ & $18$ & $\F_7$ & $32$  \\ \hline
$(38,29)$ & $38 \cdot 29$ & $18$ & $\F_9$ & $68$  \\ \hline
$(58,19)$ & $58 \cdot 19$ & $18$ & $\F_3$ & $16$  \\ \hline
\end{longtable}
}

\end{proof}

\begin{proposition}
    $X_0^{889}(1)$ does not embed in $\P^1 \times \P^1$ with bidegree $(8,10)$.
\end{proposition}

\begin{proof}

Notice that \cref{lemma: 2a_obstruction} cannot be used for $(D,N) = (889,1)$.
Since the quotient by $w_{127}$ has genus $27$, by the Castelnuovo--Severi inequality it follows that any $8$-gonal map would factor through $X_0^{889}(1)/\langle w_{127} \rangle$, hence this quotient would have geometric gonality $4$. This quotient does not have a rational point coming from a CM point, but $X_0^{889}(1)/\langle w_{889} \rangle$ does, guaranteeing that $X_0^{889}(1)/\langle w_7, w_{889} \rangle$ has a $\Q$-rational point. Then, since $X_0^{889}(1)/\langle w_7, w_{889}\rangle$ has genus $12$, it follows that 

\[
2 \le \gon_{\C}(X_0^{889}(1)/\langle w_7, w_{889} \rangle) = \gon_{\Q}(X_0^{889}(1)/\langle w_7, w_{889} \rangle) \le 4.
\]

But 

\[
(X_0^{889}(1)/\langle w_7, w_{889} \rangle)(\F_9) = 48 > 4(9+1),
\]

hence $\gon_{\Q}(X_0^{889}(1)/\langle w_7, w_{889} \rangle) > 4$, contradiction.

\end{proof}

\begin{remark}\label{remark: conjecture106}
The following (conjectural) model for $X_0^{106}(1)$ can be found in \cite{FM14}:

\[ \begin{aligned} &2x^2+12xz-6y^2-17z^2-108w^2 = 0,\\ &5x^3+33x^2z-9xy^2-45y^2z+16z^3=0. \end{aligned} \]

We note that the quadric is nonsingular, so $X_0^{106}(1)$ should embed in $\P^1 \times \P^1$ with bidegree $(3,3)$.
\end{remark}

\begin{remark} \label{remark: conjecture118}

The following is conjecturally true:

The curve $X_0^{118}(1)$ embeds in $\P^1 \times \P^1$ with bidegree $(3,3)$.

Let $K=\Q(\delta)$, where $\delta^2=-3$. Then a model for
$X_0^{118}(1)$ over $K$, in coordinates
$$
[[u_0,u_1],[v_0,v_1]]
\in
\P^1_K\times\P^1_K,
$$
is the curve $F_{3,3}=0$, where
$$
\begin{aligned}
F_{3,3}={}&
\delta\Bigl(
14178u_0^3v_0^3
-3044u_0^3v_0v_1^2
+9100u_0^2u_1v_0^2v_1
-230u_0^2u_1v_1^3\\
&\qquad
-3044u_0u_1^2v_0^3
+688u_0u_1^2v_0v_1^2
-230u_1^3v_0^2v_1
+6u_1^3v_1^3
\Bigr)\\
&+
19527u_0^3v_0^2v_1
-482u_0^3v_1^3
-19527u_0^2u_1v_0^3
+4311u_0^2u_1v_0v_1^2\\
&\qquad
-4311u_0u_1^2v_0^2v_1
+111u_0u_1^2v_1^3
+482u_1^3v_0^3
-111u_1^3v_0v_1^2.
\end{aligned}
$$

Defining equations for the genus $2$ curve $X_0^{118}/\langle w_{2} \rangle$ can be found in \cite{AH26}, while defining equations for $X_0^{118}/\langle w_{59} \rangle$ and $X_0^{118}/\langle w_{118} \rangle$, both elliptic curves over $\Q$, can be found in \cite{PS25b}. Using this information, one can conjecturally compute a model as an intersection of a cubic and a quadric with the help of ChatGPT:

$$
X_0^{118}(1)
=
\left\{
[w:x:y:z]\in \mathbf{P}^3_{\mathbf{Q}}
\;\middle|\;
\begin{array}{l}
176w^2-984wx+1371x^2+y^2+16z^2=0,\\[2mm]
(2x-w)z^2-338x^3+365wx^2-132w^2x+16w^3=0.
\end{array}
\right\}.
$$

If this model is correct, then the quadric is nonsingular, hence $X_0^{118}(1)$ embeds in $\P^1 \times \P^1$ with bidegree $(3,3)$. Moreover, the model in $\P^1 \times \P^1$ can again be conjecturally computed with the help of ChatGPT.

\end{remark}

We now prove that $X_0^{34}(7)$ does not geometrically embed in $\P^1 \times \P^1$ with bidegree $(4,4)$. We first record three facts about smooth curves of bidegree
$(4,4)$. Throughout the following lemmas, let $C\subset S:=\P^1 \times\P^1$ be such a curve, and let $p_1,p_2:S\to\P^1$
denote the two projections.

\begin{lemma}\label{lem:44-unique-extension}
Every $\sigma\in\operatorname{Aut}(C)$ extends uniquely
to an automorphism $\widetilde{\sigma}$ of $S$ preserving
$C$. Moreover, the map
\[
    \operatorname{Aut}(C)\longrightarrow\operatorname{Aut}(S),
    \qquad \sigma\longmapsto\widetilde{\sigma},
\]
is an injective group homomorphism.
\end{lemma}

\begin{proof}
Existence follows from \cref{thm: Takahashi_extension_theorem },
applied to the Hirzebruch surface $\mathcal{H}_0=S$
and a curve of bidegree $(4,4)$.
The restriction homomorphism
\[
    \operatorname{Aut}(S,C)\longrightarrow
    \operatorname{Aut}(C)
\]
is injective by the argument in the proof of
\cref{lem:fixed-point-obstruction-P1xP1}, so the extension is unique.
For $\sigma,\tau\in\operatorname{Aut}(C)$, the product
$\widetilde{\sigma}\widetilde{\tau}$ restricts to
$\sigma\tau$; uniqueness therefore gives
$\widetilde{\sigma\tau}
    =\widetilde{\sigma}\widetilde{\tau}$. Finally, $\widetilde{\sigma}=1$ implies $\sigma=1$,
proving injectivity.
\end{proof}

Composing this homomorphism with the action on the two
rulings of $S$ gives a homomorphism
$ \chi_C:\operatorname{Aut}(C)\longrightarrow
    \mathbb{Z}/2\mathbb{Z}$.
Thus $\chi_C(\sigma)=0$ precisely when
$\widetilde{\sigma}$ preserves both rulings.
We call $\sigma$ \emph{swapping} if $\chi_C(\sigma)=1$.

\begin{lemma}\label{lem:44-swapping-fixed-points}
Every swapping involution of $C$ has exactly eight
fixed points. In particular, every fixed-point-free
involution belongs to $\ker\chi_C$.
\end{lemma}

\begin{proof}
By \cref{lem:fixed-point-obstruction-P1xP1}, the extension of
a swapping involution is conjugate to the interchange
of the two factors. Its fixed locus is a curve of
bidegree $(1,1)$, which meets $C$ transversely.
Hence the number of fixed points is $(4,4)\cdot(1,1)=8$.

\end{proof}

\begin{lemma}\label{lem:44-quotient-degree-two}
Let $\sigma$ be an involution of $C$ such that
$\chi_C(\sigma)=0$ and
$\#\operatorname{Fix}_C(\sigma)>4$.
Then $C/\langle\sigma\rangle$ admits a morphism
of degree two to $\mathbb{P}^1$.
In particular, if its genus is at least two,
then it is hyperelliptic.
\end{lemma}

\begin{proof}
Write $\widetilde{\sigma}=(\alpha,\beta)$.
By \cref{lem:44-unique-extension},
$\widetilde{\sigma}$ is a nontrivial involution,
so $\alpha^2=\beta^2=1$.
If both $\alpha$ and $\beta$ were nontrivial, then
$\widetilde{\sigma}$ would have only four fixed points
on $S$, contrary to the hypothesis.
Thus exactly one of $\alpha,\beta$ is trivial.

After interchanging the factors, suppose that $\alpha=1$.
The degree-four map $p_1|_C$ is then $\sigma$-invariant,
so it factors as
\[
    C\longrightarrow C/\langle\sigma\rangle
      \xrightarrow{\,f\,}\mathbb{P}^1.
\]
Since the quotient map has degree two, $\deg(f)=2$.
\end{proof}

We now apply these lemmas to $X_0^{34}(7)$:

\begin{proposition}\label{prop:34-7-no-44-embedding}
The Shimura curve $X_0^{34}(7)$ does not geometrically
embed in $\mathbb{P}^1\times\mathbb{P}^1$
with bidegree $(4,4)$.
\end{proposition}

\begin{proof}
Put $X=X_0^{34}(7)$.
The fixed-point counts needed below are
\[
\begin{aligned}
    \#\operatorname{Fix}_X(w_2)
    &=\#\operatorname{Fix}_X(w_7)
     =\#\operatorname{Fix}_X(w_{119})=0,\\
    \#\operatorname{Fix}_X(w_{34})&=8.
\end{aligned}
\]

Suppose that $X$ admits a geometric embedding of
bidegree $(4,4)$, and let
$\chi:\Aut(X)\longrightarrow
    \Z/2\Z$
be the homomorphism associated with this embedding. By \cref{lem:44-swapping-fixed-points} we know that $w_2,w_7,w_{119}\in\ker\chi$. The Atkin--Lehner relation $w_{34}=w_2w_7w_{119}$ therefore gives $w_{34}\in\ker\chi$.

Since $\#\operatorname{Fix}_X(w_{34})=8>4$, \cref{lem:44-quotient-degree-two} implies that $Y \coloneq X/\langle w_{34}\rangle$ admits a degree-two morphism to $\mathbb{P}^1$.
However, $Y$ is a nonhyperelliptic curve of genus three
by \cite{AH26}. This is a contradiction.
\end{proof}

\begin{proposition} For $D \in \{955, 1149\}$, $X_0^{D}(1)$ does not geometrically embed in
$\P^1\times\P^1$ with bidegree $(8,10)$.
\end{proposition}

\begin{proof}
We write down the argument for $D = 955$, but the same proof also works for $D = 1149$ using
\[
    \#(X_0^{D}(1)(\F_4))=70, \quad \text{for } D \in \{955, 1149\}.
\]

Let $X \coloneq X_0^{955}(1)$. Suppose, for a contradiction, that there exists a closed immersion $\iota:X
    \hookrightarrow
    \P^1
    \times
    \P^1$ whose image has bidegree $(8,10)$.

After interchanging the two factors if necessary, let
$\pi_1 :X\longrightarrow\P^1, \pi_2 :X\longrightarrow\P^1$
be the two projections, with $\deg(\pi_1)=8,\deg(\pi_2)=10$.

Apply the Tower~\cref{thm: tower-theorem} to the degree-$8$ morphism $\pi_1$.
There exists a smooth projective curve $Y$ defined over $\Q$
and a morphism $\pi:X\longrightarrow Y$ defined over $\Q$, of degree $d'$ dividing $8$, such that
\[
    g(Y)\le
    \left(\frac{8}{d'}-1\right)^2.
\]
Thus $d'\in\{1,2,4,8\}$.
We first eliminate all possibilities except $d'=2$.

If $d'=1$, then $\pi$ is an isomorphism, since both curves are
smooth and projective. Therefore $63=g(X)=g(Y)\le (8-1)^2=49$, which is impossible.

Suppose that $d'=4$. Then
\[
    g(Y)\le \left(\frac84-1\right)^2=1.
\]
If $g(Y)=1$, the good reduction \cref{lemma : good_reduction_lemma} gives, after reduction at $2$,
a degree-$4$ morphism $\widetilde\pi:\widetilde X\longrightarrow\widetilde Y$ defined over $\F_2$. Hence $\#\widetilde X(\F_4)
    \le
    4\,\#\widetilde Y(\F_4)$.
By the Hasse--Weil bound,
\[
    \#\widetilde Y(\F_4)
    \le
    4+1+2g(Y)\sqrt4
    \le 9.
\]
Consequently $70
    =\#\widetilde X(\F_4)
    \le 4\cdot 9=36$, a contradiction.

If instead $g(Y)=0$, the genus-zero part of the good reduction \cref{lemma : good_reduction_lemma} gives a morphism from $\widetilde X$ to a rational curve over
$\F_2$ of degree at most $4$. Since a smooth rational curve
over a finite field is isomorphic to $\P^1$, we obtain
\[
    70
    \le 4\,\#\P^1(\F_4)
    =4(4+1)=20,
\]
again a contradiction. Thus $d'\neq4$.

If $d'=8$, then
\[
    g(Y)\le
    \left(\frac{8}{8}-1\right)^2=0.
\]
The same genus-zero part of the good reduction \cref{lemma : good_reduction_lemma} gives a
morphism of degree at most $8$ from $\widetilde X$ to a rational
curve over $\F_2$. Therefore
\[
    70
    \le 8\,\#\P^1(\F_4)
    =8 \cdot 5 =40,
\]
again impossible.

We conclude that
\[
    \deg(\pi)=2
    \qquad\text{and}\qquad
    g(Y)\le
    \left(\frac82-1\right)^2=9.
\]

We now use the Castelnuovo--Severi inequality. Consider the morphism
\[
    (\pi,\pi_1):
    X
    \longrightarrow
    Y\times\P^1,
\]
let $Z$ be the normalization of its image, and let
\[
    h:X\longrightarrow Z
\]
be the induced morphism, so that $\pi$ and $\pi_1$ both factor through $h$:
\[
    \pi = p\circ h,
    \qquad
    \pi_1 = q\circ h,
\]
where $p:Z\to Y$ and $q:Z\to\P^1$ are induced by the two
projections. In particular
\[
    2=\deg(\pi)=\deg(h)\deg(p).
\]

Suppose first that $\deg(h)=1$, i.e.\ $(\pi,\pi_1)$ is birational onto its
image. Since $\deg(\pi)=2$ and $\deg(\pi_1)=8$, Castelnuovo--Severi gives
\[
    g(X)
    \le
    2g(Y)+8g(\P^1)+(2-1)(8-1).
\]
Thus $63 \le 2\cdot9+7=25$, which is impossible.

Hence $\deg(h)=2$ and $\deg(p)=1$. Since $Z$ and $Y$
are smooth projective curves, $p$ is an isomorphism. Setting $f_Y = q\circ p^{-1}:
    Y\longrightarrow\P^1$,
we obtain
\[
    \pi_1
    = q\circ h
    = f_Y\circ p\circ h
    = f_Y\circ\pi,
\]
so $\pi_1$ factors through $\pi$.

Apply the same argument to the degree-$10$ projection $\pi_2$: let $Z'$ be
the normalization of the image of $(\pi,\pi_2)$, with induced morphism
$h':X\to Z'$, so that $\deg(h')\deg(p')=2$ for the
induced morphism $p':Z'\to Y$. If $\deg(h')=1$, then
$(\pi,\pi_2)$ is birational onto its image and Castelnuovo--Severi gives
\[
    63
    \le
    2g(Y)+(2-1)(10-1)
    \le
    18+9
    =27,
\]
again impossible. Hence $\deg(p')=1$, so $p'$ is an isomorphism and,
as before, $\pi_2$ factors through $\pi$:
\[
    \pi_2 = g_Y\circ\pi,
    \qquad
    g_Y:Y\longrightarrow\P^1.
\]

Consequently $\iota=(\pi_1,\pi_2)
    =(f_Y,g_Y)\circ\pi$.
Thus $\iota$ factors through the degree-$2$ morphism $\pi$, and
therefore cannot be a closed immersion. This contradiction proves
that no geometric embedding of bidegree $(8,10)$ exists.
\end{proof}

\begin{proposition}
We have the following:
\begin{itemize}
    \item $X_0^{62}(5)$ and $X_0^{94}(3)$ do not geometrically embed in $\P^1 \times \P^1$ with bidegree $(4,6)$,
    \item $X_0^{51}(7)$ and $X_0^{301}(1)$ do not geometrically embed in $\P^1 \times \P^1$ with bidegree $(4,8)$,
    \item $X_0^{694}(1)$ does not geometrically embed in $\P^1 \times \P^1$ with bidegree $(5,8)$.
\end{itemize}
\end{proposition}

\begin{proof}
We use \cref{lemma: 2a_obstruction} together with the following information:

\[
\begin{array}{c|ccc}
(D,N) & r_{m_1} & r_{m_2} & r_{m_3}\\
\hline
(62,5) & r_{155}=8 & r_{310}=8 & r_2=4,\\
(94,3) & r_{141}=8 & r_{282}=8 & r_2=4, \\
(51,7) & r_3 = 8 & r_7 = 8 & r_{21} =0 , \\
(301,1) & r_{43}=8 & r_{301}=8 & r_7=0 , \\
(694,1) & r_{347}= 10 & r_{694}= 10 & r_2= 2.
\end{array}
\]

\end{proof}

\begin{proposition} \label{proposition: not_in_P2}
$X_0^{D}(N)^+$ does not geometrically embed in $\P^2$ and $X_0^{D}(N)^+(\Q) \ne \emptyset$ for $(D,N) \in \{ (6,157),(622,1),(21,32),(38,31),(39,23),(745,1), (1365,1) \}$. 
\end{proposition}

\begin{proof}

Using the code from \cite{MPSS25} we prove that all quotients by the Fricke involution have $\Q$-rational points coming from images of CM points. 

$X_0^{6}(157)^+$ has genus $10 =\frac{(6-1)(6-2)}{2}$.
By \autoref{theorem: smooth_plane_fixed_points}, any quotient of a plane sextic by an involution would have genus $4$, while $X_0^{6}(157)/\langle w_2, w_{471} \rangle$ has genus $5$.

Similarly,

\[
g(X_0^{622}(1)^+) = 10, \quad g(X_0^{622}(1)^*) = 5.
\]

$X_0^{745}(1)^+$ has genus $21 =\frac{(8-1)(8-2)}{2}$. By \autoref{theorem: smooth_plane_fixed_points}, any quotient of a plane octic by an involution would have genus $9$, while $X_0^{745}(1)^*$ has genus $11$.

Similarly,

\[ g(X_0^{21}(32)^+) =21, \quad 
g(X_0^{21}(32)/\langle w_{21},w_{32} \rangle) = 11,
\]

\[ g(X_0^{38}(31)^+) =21, \quad 
g(X_0^{38}(31)/\langle w_{38},w_{31} \rangle) = 11,
\]

\[ g(X_0^{39}(23)^+) =21, \quad 
g(X_0^{39}(23)/\langle w_{39},w_{23} \rangle) = 11,
\]

\[ g(X_0^{1365}(1)^+) =21, \quad 
g(X_0^{1365}(1)/\langle w_3,w_{1365} \rangle) = 11.
\]

\end{proof}

\begin{proposition}\label{prop:replacement-54}
For $(D,N)\in\{(6,157),(622,1)\}$, the Shimura curve $X_0^D(N)$ does
not geometrically embed in $\P^1\times\P^1$ with bidegree $(6,6)$.
For
\[
  (D,N)\in
  \{(21,32),(38,31),(39,23),(745,1),(1365,1)\},
\]
it does not geometrically embed in $\P^1\times\P^1$ with bidegree
$(8,8)$.
\end{proposition}

\begin{proof}
Put
\[
  X \coloneq X_0^D(N),\qquad Y \coloneq X_0^D(N)^+,
\]
and choose $n$ according to the table below. The Fricke quotient
$X\to Y$ has degree two, with
\[
  g(X)=(n-1)^2,
  \qquad
  g(Y)=\frac{(n-1)(n-2)}{2}.
\]
By \cref{proposition: not_in_P2}, $Y(\Q)\neq\varnothing$ and $Y$ admits no geometric
smooth plane model.

Suppose that $X$ geometrically embeds in $\P^1\times\P^1$ with
bidegree $(n,n)$. \cref{prop:balanced-double-quotient}
then implies $\gon_{\C}(Y)\leq n/2$.
Since $n/2\leq 4$ and $g(Y)\geq 10$, \cref{corollary: low_gon_eq} gives
$\gon_{\Q}(Y)=\gon_{\C}(Y)\leq n/2$.
The point counts in the following table contradict this inequality
in every case:
\begin{center}
\begin{tabular}{rrrrrr}
\toprule
$(D,N)$ & $n$ & $g(Y)$ & $q$ & $\#Y(\F_q)$ & $\frac n2(q+1)$ \\
\midrule
$(6,157)$  & $6$ & $10$ & $5$  & $22$  & $18$  \\
$(622,1)$  & $6$ & $10$ & $9$  & $40$  & $30$  \\
$(21,32)$  & $8$ & $21$ & $25$ & $120$ & $104$ \\
$(38,31)$  & $8$ & $21$ & $9$  & $52$  & $40$  \\
$(39,23)$  & $8$ & $21$ & $4$  & $36$  & $20$  \\
$(745,1)$  & $8$ & $21$ & $9$  & $64$  & $40$  \\
$(1365,1)$ & $8$ & $21$ & $4$  & $36$  & $20$  \\
\bottomrule
\end{tabular}
\end{center}
In each row the characteristic of $\F_q$ does not divide $DN$, so
$X$ has good reduction there. \cref{lemma : good_reduction_lemma} gives good reduction for $Y$,
and \cref{NO-Fp} yields $\gon_{\Q}(Y)>n/2$, a contradiction.
\end{proof}

\section{Shimura curves $X_0^D(N)$ that admit a smooth plane model}

Assume $X_0^D(N)$ is a Shimura curve that admits a degree $d$ smooth plane model over $\C$. We know that the geometric gonality of a smooth plane curve of degree $d$ is $d-1$, see \cite[Theorem A]{CK90}, while the genus-degree formula states that $g(X_0^D(N)) = \frac{(d-1)(d-2)}{2}$. Then using \autoref{theorem: genus_gonality_inequality} we obtain:
\[
\frac{1}{2} \cdot \frac{975}{4096} \left( \frac{(d-1)(d-2)}{2}
-1\right) = \frac{1}{2} \cdot \frac{975}{4096} ( g(X_0^D(N))-1) \le \gon_{\C} (X_0^D(N)) = d-1,
\]
which implies that $d \le 18$. Thus the genus of $X_0^D(N)$ is at most $136$.

By \cref{lemma: genus_upper_bound}, we find that for $DN > 99049$ we have $g(X_0^D(N)) > 136$. Therefore, we only need to consider the pairs $(D,N)$ where $DN \le 99049$. We may consider only those pairs for which $\omega(D)$ is even and $\gcd(D,N) = 1$, as the set of ramified places has even cardinality. We further need that 
\[
g(X_0^D(N)) \in \left \{ \frac{(d-1)(d-2)}{2} : 2 \le d \le 18  \right \}.
\]
Thus we are left to analyze $274$ pairs $(D,N)$.
Let $(D,N)$ be such a pair, for which \[g(X_0^D(N)) = \frac{(d-1)(d-2)}{2}\] for some $d \in \{ 2,3, \dots, 18 \}$.

On one hand, by \autoref{thm: Ogg_fixed_pts}, the fixed points of $w_m$ have CM by the indicated quadratic orders, and their number $r_m \coloneq \#\Fix_{X_0^D(N)}(w_m)$ is obtained by summing the corresponding local-embedding contributions. On the other hand, if $X_0^D(N)$ admits a plane model of degree $d \ge 4$, then by \autoref{theorem: smooth_plane_fixed_points} we have that:
\begin{equation}\label{eqn: equal_genera}
\#\Fix_{X_0^D(N)}(w_m) = d + \frac{1-(-1)^d}{2} \quad \text{for all } \, 1 < m \parallel DN.
\end{equation}
For $260$ out of the $274$ candidate pairs $(D,N)$ we have that $g(X_0^D(N)) \ge 3$, while for the remaining $14$ we have $g(X_0^D(N)) \le 1$.
Running our code on these $260$ pairs $(D,N)$ we find that \cref{eqn: equal_genera} does not hold for any of the curves with $g(X_0^D(N)) \ge 3$. Thus there is no Shimura curve $X_0^D(N)$ admitting a smooth plane model of degree $d \ge 4$.
In conclusion,  $X_0^D(N)$ admits a smooth plane model if and only if  $g(X_0^D(N)) \le 1$.

\bibliographystyle{amsalpha}
\bibliography{biblio}

\end{document}